\documentclass[10pt]{amsart}

\usepackage{graphicx}

\theoremstyle{plain}
\newtheorem{theorem}{Theorem}[section]
\newtheorem{corr}[theorem]{Corollary}
\newtheorem{lem}[theorem]{Lemma}
\newtheorem{conj}[theorem]{Conjecture}
\newtheorem{prop}[theorem]{Proposition}

\theoremstyle{definition}
\newtheorem{defi}[theorem]{Definition}
\newtheorem{rem}[theorem]{Remark}

\def\Gauss diagram{\Delta}

\def\R{{\mathbb R}}

\def\1{\hbox{\rm\rlap {1}\hskip.03in{\rom I}}}
\def\Bbbone{{\rm1\mathchoice{\kern-0.25em}{\kern-0.25em}
        {\kern-0.2em}{\kern-0.2em}I}}

\let\reference\ref

\def\bt{\bar t_2}

\let\ea\expandafter
\let\nx\noexpand
\makeatletter

\def\vcbox#1{\setbox\@tempboxa=\hbox{#1}\parbox{\wd\@tempboxa}{\box
  \@tempboxa}}

\def\@test#1#2#3#4{%
  \let\@tempa\go@
  \@tempdima#1\relax\@tempdimb#3\@tempdima\relax\@tempdima#4\unitxsize\relax
  \ifdim \@tempdimb>\z@\relax
    \ifdim \@tempdimb<#2%
      \def\@tempa{\@test{#1}{#2}}%
    \fi
  \fi
  \@tempa
}

\def\go@#1\@end{}
\newdimen\unitxsize
\newif\ifautoepsf\autoepsftrue
\autoepsffalse

\def\epsfsize#1#2{\epsfxsize\relax\ifautoepsf
  {\@test{#1}{#2}{0.1 }{4   }
                {0.2 }{3   }
                {0.3 }{2   }
                {0.4 }{1.7 }
                {0.5 }{1.5 }
                {0.6 }{1.4 }
                {0.7 }{1.3 }
                {0.8 }{1.2 }
                {0.9 }{1.1 }
                {1.1 }{1.  }
                {1.2 }{0.9 }
                {1.4 }{0.8 }
                {1.6 }{0.75}
                {2.  }{0.7 }
                {2.25}{0.6 }
                {3   }{0.55}
                {5   }{0.5 }
                {10  }{0.33}
                {-1  }{0.25}\@end
                \ea}\ea\epsfxsize\the\@tempdima\relax
                \fi
                }

\def\CD#1{{\let\@nomath\@gobble\small\diag{6mm}{2}{2}{
  \pictranslate{1 1}{
    \piccircle{0 0}{1}{}
    #1
}}}}

\def\GD#1{{\let\@nomath\@gobble\scriptsize\diag{6mm}{3.0}{3.0}{
  \pictranslate{1.5 1.5}{
    #1
    \piccircle{0 0}{1}{}
}}}}

\def\GDD#1{{\let\@nomath\@gobble\scriptsize\diag{9mm}{3.0}{3.0}{
  \pictranslate{1.5 1.5}{
    #1
    \piccircle{0 0}{1}{}
}}}}

\begin{document}
\hyphenation{Ca-m-po}
\title[
Low complexity algorithms in knot theory
] 
{Low complexity algorithms in knot theory}

\author[O.~Kharlampovich, A.~Vdovina]{Olga Kharlampovich and Alina Vdovina}

\address{}

\email{}

\address{}

\email{}

\begin{abstract}

We show that the  genus problem for  alternating knots   with $n$ crossings has linear time complexity and is in Logspace$(n)$.  
 Almost all alternating knots of given genus  possess additional combinatorial
 structure, we call them standard. We show that the genus problem for these knots belongs to $TC^0$ circuit complexity class. We also show, that the  equivalence problem  for such knots with $n$ crossings  has time complexity $n\log (n)$ and is in Logspace$(n)$ and $TC^{0}$ complexity classes.

\end{abstract}

\maketitle

\section{Introduction}

Determining whether a given knot is trivial or not is one of the  central
questions in topology.  Dehn's work \cite{Dehn}  led to the
formulation of the word and isomorphism problems, which played the major 
role in the development of the theory of algorithms. The first algorithm for the
unknotting problem was given by Haken \cite{Haken}. HakenÕs procedure is based on normal
surface theory.   Hass,
Lagarias and Pippenger showed that HakenÕs unknotting algorithm runs in time at
most $C^t$
where the knot $K$ is embedded in the 1-skeleton of a triangulated manifold
$M$ with $t$ tetrahedra, and $C$ is a constant independent of $M$ and  $K$, see \cite {LP}.  They also showed that
the unknotting problem is in NP. Agol, Haas and Thurston  \cite{AHT} showed that the problem of determining a bound on the genus of a knot in a 3-manifold, is NP-complete.  For more details on NP and co-NP problems in knot theory see an excellent survey by Lackenby \cite{Lackenby1}. This shows that  (unless P=NP) the genus problem has high computational complexity even for knots  in a 3-manifold. 

In this paper we initiate the study of classes of knots where the genus problem and even the equivalence problem have very low computational complexity.  
We show that the  genus problem for  alternating knots   with $n$ crossings has linear time complexity and is in Logspace$(n)$.  
 Almost all alternating knots of given genus  possess additional combinatorial
 structure, we call them standard. We show that the genus problem for these knots belongs to $TC^0$ circuit complexity class. We also show, that the  equivalence problem  for such knots with $n$ crossings  has time complexity $n\log (n)$ and is in Logspace$(n)$ and $TC^{0}$ complexity classes.
 
 Recall, that
$AC^0$
($TC^0$
) is the class of functions computed by constant depth boolean
circuits of unbounded fan-in AND, OR, and NOT gates (MAJORITY gates,
respectively). The relationship is as follows $TC^0\subseteq L\subseteq P$, where $P$ is the class of polynomial time problems. 

The main tool applied in our algorithms are
quadratic words in a free group. Such words with some additional structure are  called Wicks forms.
It was shown in \cite{SV} that the alternating diagrams obtained from
planar maximal Wicks forms are standard alternating, and
an unexpected consequence of this result is
that generically an alternating knot of any genus 
(higher than one) is a standard alternating knot.

For  basic knot theoretic definitions see \cite{Rolfsen}.

\section{Statement of results and structure of the paper}

Seifert algorithm if a standard tool to associate an orientable surface with a boundary to a knot, which can be found in, for example \cite{Rolfsen},
but we remind it here for the completeness of the arguments.

Seifert's Algorithm(1934, Herbert Seifert):
\begin{itemize}
\item Input := a knot $K$.
 \item Output := an orientable surface $S_k$ with boundary component $K$.
 \end{itemize}
 
Algorithm:
\begin{enumerate}
\item Start with a diagram of $K$.
\item Give it an orientation.
\item Eliminate the crossings as follows: First note that at each crossing two strands come in and two come out. Then connect each of the strands coming into the crossing to the adjacent strand leaving the crossing.
\item Fill in the circles, so each circle bounds a disk.
\item Connect the disks to one another, at the crossings of the knots, by twisted bands.

\end{enumerate}

\begin{defi}
The oriented circles appearing in the Seifert's Algorithm
are called Seifert Circles.
\end{defi}

We note, that
by changing the orientation we get the same Seifert Circles but with opposite
orientation. Therefore the result is independent of the orientation.

\begin{defi}
For a diagram $D$ of knot $K$, we define the {\em genus} $g(D)$ as the
genus of the surface obtained by applying the Seifert algorithm to
this diagram. It can be expressed as
\[
g(D)=\frac{c(D)-s(D)+1}{2}\,,
\]
with $s(D)$ being the number of Seifert circles of $D$. 
\end{defi}

\begin{defi}
A planar diagram is alternating if the over-crossings and under-crossings alternate.
A knot is alternating if it has an alternating diagram. 
\end{defi}

\begin{theorem}\label{th:genus}
The  genus problem for  alternating knots with $n$ crossings has linear time complexity and is in Logspace$(n)$. For an arbitrary knot diagram there is an algorithm with the same complexity that determines the genus of the diagram.   The genus problem for  standard (see Definition 4.9) alternating knots with $n$ crossings is in 
 $TC^{0}$ complexity class.
\end{theorem}

\begin{theorem}\label{th:iso}
The isomorphism problem for standard  alternating knots with $n$ crossings  has time complexity $n\log (n)$ and is in Logspace$(n)$ and $TC^{0}$ complexity classes.

\end{theorem}

The paper is organized as follows: we start with some relevant definitions and known facts in Section 3, Section 4 includes explanations why almost all alternating knots are standard.
Section 5 describes an algorithm of getting a standard alternating knot using a Bieulerian path
in a 3-connected planar 3-valent graph. Section 6 defines extended Wicks forms.
Sections 7 and 8 bring all facts together to prove the main results.

\section{Preliminaries and genus problem for alternating knots}

We start with some classical definitions and recall important
properties of alternating knots and links.

\begin{defi}
A crossing $p$ in a knot diagram $D$ is called {\em reducible}
(or nugatory) if $D$ can be represented in the form
\begin{figure}[htbp]
\includegraphics[scale=0.7]{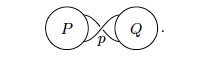}
\end{figure}

$D$ is called reducible if it has a reducible crossing, else it is
called {\em reduced}.
\end{defi}

\begin{defi}
Denote by $c(D)$ the {\em crossing number} of a knot diagram $D$. The
{\em crossing number} $c(K)$ of a knot $K$ is the minimal crossing
number $c(D)$ of all diagrams $D$ of $K$.
\end{defi}

\begin{theorem}({\cite{Kauffman,Murasugi2,Thistle})}
An alternating knot with a reduced alternating diagram of
$n$ crossings has crossing number $n$.
\end{theorem}

\begin{defi}
For a diagram $D$ of knot $K$, we define the {\em genus} $g(D)$ as the
genus of the surface obtained by applying the Seifert algorithm to
this diagram. It can be expressed as
\[
g(D)=\frac{c(D)-s(D)+1}{2}\,,
\]
with $s(D)$ being the number of Seifert circles of $D$. 
\end{defi}

The importance of this definition relies on the following classical
fact:

\begin{theorem}{\rm (\cite{Crowell,Murasugi})}
An alternating knot with an alternating diagram of genus $g$
has genus $g$.

\end{theorem}

Theorems 3.3 and 3.5 imply that to determine the genus and the crossing number of an alternating knot it is sufficient to consider its
alternating diagram.  It has the same genus and  crossing number.

Knots diagrams give rise to quadratic words in the following way.
 
Knots (smooth embeddings of ${\mathbb S}^1$ to ${\mathbb R}^3$) are usually presented
by knot diagrams that are generic immersions of ${\mathbb S^1}$ to 
${\mathbb R}^2$-plane enhanced by information of over-passes and under-passes at the double points. To correspond a quadratic word to a knot diagram $D$ one assigns a letter 
to each double point of the immersion, and the preimages of each double point are denoted by this letter with opposite exponents, 1 and -1.

Our algorithm of computing the genus of an alternating diagram is based on 
the fact that the genus of an alternating diagram and the corresponding quadratic word
coincide, what is shown by the following theorem.

\begin{theorem}
The genus of an alternating diagram is the same as  the genus of the corresponding  quadratic word.
\end{theorem}

\begin{proof}
By the Theorem 3.5 the genus of an alternating knot $K$ is equal to the genus of an alternating
diagram of $K$. It was shown in \cite{STV} that the genus of an alternating diagram is equal to the genus
of the corresponding quadratic word.
\end{proof}

\section{Introduction to standard knots}

By the work of Menasco and Thistlethwaite \cite{MenThis}, alternating
knots are  related to  diagrammatic move called flype.

\begin{defi}
A {\em flype} is a move on a diagram shown in figure \reference{fig1}.

\begin{figure}[htbp]
\includegraphics[scale=0.7]{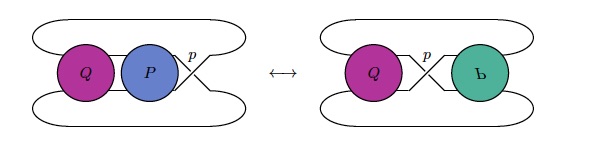}
\caption{A flype near the crossing $p$\label{fig1}}
\end{figure}

\begin{theorem} {\rm (\cite{MenThis})} \label{TMT}
Two alternating diagrams of the same knot or link are flype-equivalent,
that is, transformable into each other by a sequence of flypes.
\end{theorem}

When we want to specify the distinguished crossing $p$,
we say that it is a flype {\em near} the crossing $p$.

We call the tangle $P$ of figure \reference{fig1} {\em flypable}.
We say that the crossing $p$ {\em admits a flype} or that {\em the
diagram admits a flype at (or near) $p$}.

We call the flype {\em non-trivial}, if both tangles $P$ and $Q$ have at
least two crossings.

We say that the crossing $p$ {\em admits} a (non-trivial) flype
if the diagram can be represented as in figure \reference{fig1} with
$p$ being the distinguished crossing (and both tangles having
at least two crossings). A diagram admits a (non-trivial) flype
if some crossing in it admits a (non-trivial) flype.

Since trivial flypes are of no interest we will assume from now on,
unless otherwise noted, that all flypes are non-trivial, without
mentioning this explicitly each time.
\end{defi}

 We call the move in (1) a $\bt$ move.
\begin{theorem}{\rm (\cite[theorem 3.1]{gen1})}\label{th1}
Reduced (that is, with no nugatory crossings)
alternating knot diagrams of given genus
decompose into finitely many equivalence classes
under flypes and direct and reversed applications
of $\bt$ moves.
\begin{figure}[htbp]
\includegraphics[scale=0.7]{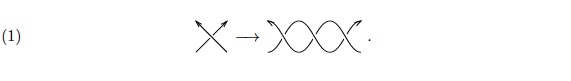}
\end{figure}\end{theorem}

$\bt$-irreducible
diagram is a diagram where we cannot reduce the number of crossings using $\bt$-moves.
 
It was observed in \cite{gen1} that in a sequence of
flypes and $\bt$ moves, all the flypes can be performed in
the beginning. It follows then from Theorems \ref{TMT} and \ref{th1} that there
are only finitely many alternating knots with $\bt$-irreducible
diagrams of given genus $g$, and we call all such knots, and
their alternating diagrams {\em generators} or {\em generating}
knots/diagrams of genus $g$.

A {\em clasp} is a tangle made up of two crossings. According to the
orientation of the strands we distinguish between reverse and
parallel clasps.
\begin{figure}[htbp]
\includegraphics[scale=0.7]{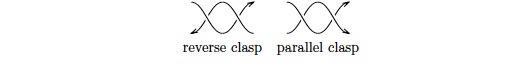}
\end{figure}
There is an obvious bijective correspondence between the
crossings of the 2 diagrams in figure \reference{fig1} before and after the flype,
and under this correspondence we can speak of what is
a specific crossing after the flype. In this sense, we 
make the following definition:

\begin{defi}
We call two crossings in a diagram $\sim$-{\em equivalent}, if they
can be made to form a reverse clasp after some (sequence of) flypes.
\end{defi}

Is is an easy exercise to check that $\sim$ is an equivalence
relation.

\begin{defi}
We call an alternating diagram {\em generating}, if each $\sim$-%
equivalence class of its crossings has $1$ or $2$ elements. The set of
diagrams which can be obtained by applying flypes and $\bt$ moves on a
generating diagram $D$ we call {\em generating series} of $D$.
\end{defi}

Thus theorem \reference{th1} says that alternating diagrams of given
genus decompose into finitely many generating series.

\begin{defi}\label{df1}
Let $c_g$ be the maximal crossing number of a generating diagram
of genus $g$, and $d_g$ the maximal number of $\sim$-equivalence
classes of such a diagram.
\end{defi}

\begin{theorem}\label{th2}[\cite{SV}]
The following holds:
\begin{enumerate}
\item $d_{g,o}=d_{g,e}=6g-3$ for $g>1$. That is, $a_{n,g}\sim n^{6g-4}$.
\item $c_g\ge 10g-7$.
\end{enumerate}
\end{theorem}

It will be convenient, from now on, to consider only genus $g>1$.
The case $g=1$ is described in \cite{gen1}.

\begin{defi}
We say, that an alternating knot diagram is strongly prime, if it admits the maximal number
of $\sim$-equivalence
classes.
\end{defi}

\begin{defi}
A {\em standard} diagram {\bf D} of an alternating knot is a strongly 
prime diagram  all of whose Seifert circles
have either an empty interior or exterior and each of the Seifert circles of {\bf D}
has 2 or 3 adjacent crossings.(Here interior and
exterior denote the bounded and unbounded connected component of
the complement of the Seifert circle in $\R^2$ and empty means
not containing a crossing of the knot diagram.)
\end{defi}

\begin{defi}
A knot admitting a standard knot diagram is called standard knot. 
\end{defi}

We consider planar 3-connected 3-valent graphs (with no multiple
edges and loops). When equipping such a graph with a
Bieulerian path (whenever this is possible),
we associate to it a standard generating knot .

As a Bieulerian path endows
each vertex of such a graph with a cyclic orientation,
we have yet another appearance of, at least some, 3-valent
graphs from the theory of Vassiliev invariants \cite{Bar-Natan2}
in a different 
context, after Bar-Natan's remarkable paper \cite{Bar-Natan}.

A consequence of such a correspondence is that  standard alternating knots
dominate among alternating knots of given genus (higher than 1),
as the crossing number increases. The theorem below is a slight modification of
results of \cite{SV}, but we prove it at the end of Chapter 5 for the completeness of the paper.

\begin{theorem}\label{th3}  The family of standard alternating knots is generic in the family of all alternating knots, namely the ratio of the cardinality of the set
of standard alternating knots $K$ with $c(K)=n, g(K)=g,$ to the cardinality of the set of all alternating knots of the same genus and crossing number
approaches 1
as $n\to\infty$ for any fixed $g>1$.
\end{theorem}

In \cite{VirPol} the concept of Gau\ss{} diagrams was introduced
as a tool for generating knot invariants. Given a knot
diagram, one links by a chord on a circle the preimages
of the two passes of each crossing, orienting the chord
from the underpass to the overpass. The resulting object is
called a {\em Gau\ss{} diagram} (GD).

In general any circle with oriented chords is called
a Gau\ss{} diagram. Not all Gau\ss{} diagrams come from knot
diagrams; those that do are called {\em realizable} Gau\ss{} diagrams.
We ignore in the sequel the sign of the crossings, that is, the
direction of the arrows. Then realizable Gau\ss{} diagrams correspond
bijectively to alternating knot diagrams up to mirroring.
It was notices in \cite{STV} that the Gau\ss{} diagram of a
generating diagram has no triple of chords, not intersecting
each other, and intersecting the same subset of
the remaining chords.

\section{Standard alternating knots and 3-valent graphs\label{IMGD}}

Let $G$ be a connected 3-valent graph. Fix some arbitrary orientation
(direction) of the edges in $G$.
A {\em Bieulerian path} in $G$ is a closed path that traverses each
edge of $G$ exactly twice, only once in each direction, and does not
traverse any edge followed immediately by its inverse
(itself in the opposite direction).

To a Bieulerian path one can associate a word in some alphabet
(called {\em Wicks form} and considered in more detail later),
obtained by labeling each edge by a letter,
and putting this letter (resp. its inverse) when the edge is
traversed in (resp.\ oppositely to) its orientation.


In \cite{STV} we described the bijection between a 
graph with Bieulerian path $G$ and a Gau\ss{} diagram $G'$, as the following.

To obtain $G'$ from $G$ one just writes the letters of its
word (Wicks form) $w$ along a circle and links by a chord each letter
and its inverse. To obtain $G$ from $G'$, we consider the
circle of $G'$ as a $2n$-gon (each side corresponding to
a basepoint of a chord) and identify sides corresponding to
the basepoints of the same chord, obtaining $G$ lying on
a surface $S$. (The circle $G'$ bounds a disk that yields
$S$ under the identifications.) To indicate the origin of $G$ and $S$,
we write $G=G(w)$ and $S=S(w)$. The dual of $G$ forms
a 1-vertex triangulation of $S$.

We call a graph with a Bieulerian path {\em realizable} if and only if
its associated Gau\ss{} diagram is realizable (as a knot diagram).
In this case each Seifert circle of the knot diagram corresponds to
a vertex of the graph, and each crossing of the knot diagram
attached to a pair of Seifert circles
corresponds to an edge joining the vertices of these Seifert circles.
In this sense we call the number of crossings attached a
Seifert circle its {\em valence} (the
valence of its corresponding vertex in the graph).

Then in \cite{STV} we defined the genus of Gau\ss{} diagrams
and of graphs in different ways and showed that they
coincide. Also the genus of a knot diagram (which is
equal for alternating diagrams to the genus of the knot
\cite{Crowell,Murasugi}) was shown to
be equal to the genus of its Gau\ss{} diagram.

It is easy to see that composite knot diagrams give composite
Gau\ss{} diagrams, which in turn correspond to graphs with a cut
vertex. Since genus is additive under the join of graphs
\begin{figure}[htbp]
\includegraphics[scale=0.7]{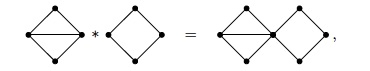}
\end{figure}

as mentioned, a composite genus $g$ knot diagram can have at most $6g-6$
$\sim$-equivalence classes. Thus the contribution of such
diagrams is negligible, once we have shown that there are
diagrams with more $\sim$-equivalence classes (see the proof
of theorem \reference{th2}).

\begin{defi}
A primitive Conway tangle \cite{Conway} is a tangle of the form
\begin{figure}[htbp]
\includegraphics[scale=0.7]{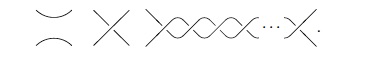}
\end{figure}

We call two crossings $a$ and $b$ in a diagram $D$ {\em neighbored}, if
they belong to a reversely oriented primitive Conway tangle in $D$, that
is, there are crossings $c_1,\dots,c_n$ with $a=c_1$ and $b=c_n$,
such that $c_i$ and $c_{i+1}$ form a reverse clasp in $D$.
(Equivalently, $a$ and $b$ correspond in the graph to edges
which can be connected by a path passing only through vertices
of valence 2.)
\end{defi}

This is a similar definition to $\sim$-equivalence, but with no flypes
allowed. Thus the number of $\sim$-equivalence classes of a diagram is
not more than the number of neighbored equivalence classes of
the same diagram, or of any flyped version of it.

The following was proved in \cite{SV}.

\begin{lem}\label{lem23}
A knot diagram of genus $g$ has at most $6g-3$ neighbored equivalence
classes (and hence at most $6g-3$ $\sim$-equivalence classes).

Moreover, knot diagrams of genus $g$ having exactly $6g-3$ neighbored
equivalence classes come exactly
from graphs with Bieulerian path, all whose vertices have valence
2 or 3.
\end{lem}

The lemma means in particular, that if $G'$ is realizable and
its knot diagram $D$ has $6g-3$ $\sim$-equivalence
(or just neighbored equivalence)
classes, then all vertices of $G'$ have valence $2$ or $3$, and thus the
Seifert circles of $D$ have $2$ or $3$ adjacent crossings.
Hence the knot diagram is standard.

In general the condition of being realizable is difficult to test
for $G'$, but in the trivalent case it is surprisingly simple.

\begin{theorem}[\cite{SV}]
A trivalent graph with Bieulerian path is realizable if and only if
it is planar(ly embeddable). In this case the knot diagram is standard.
\end{theorem}

We should remark that a planar graph is in fact a graph equipped with
a concrete planar embedding, while the realizability of the graph
does not depend on the planar embedding. However, we will shortly
show that for the cases we need to consider the planar embedding
is unique (see remark \reference{remb}).

For the proof, and later, we will need the following additional
structure on a trivalent graph with Bieulerian path.

\begin{defi}
A Bieulerian path in a trivalent graph
induces an orientation on each 3-valent vertex $v$ given by a
cyclic order of the 3 adjacent edges. To define it, orient
the 3 adjacent edges $a$, $b$ and $c$ towards $v$. Then
if the word of the Bieulerian path contains the subwords $ab^{-1}$,
$bc^{-1}$ and $ca^{-1}$ (in whatever order), then the orientation at
$v$ is given by $(a,b,c)$.
\begin{figure}[htbp]
\includegraphics[scale=0.5]{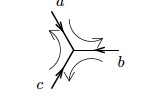}
\end{figure}

If the Bieulerian path contains the subwords $ac^{-1}$,
$cb^{-1}$ and $ba^{-1}$ (in whatever order), then the orientation at
$v$ is $(c,b,a)$.
\begin{figure}[htbp]
\includegraphics[scale=0.5]{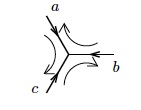}
\end{figure}

\end{defi}

Now we establish
a natural correspondence between a planar $3$-valent graph with
Bieulerian path and a standard knot diagram.

We give an explicit construction of a standard knot diagram using a 
planar $3$-valent graph with Bieulerian path.

Let $G$ be a 3-valent graph with Bieulerian path.
The path induces the orientation of vertices. If two ends
of the edge have the same orientation, put on the edge an additional
vertex of degree two. We have a graph $G^{\prime}$ with vertices
of degree two and three. Every edge $x$ of $G$, which was
divided in two parts, will be replaced in the Bieulerian path
by $x_1x_2$. We can change the orientations of the
edges of $G^{\prime}$ such that in the Bieulerian path
the orientations of edges alternate. Now we have an oriented
graph such that for every vertex all edges incident to it
either all are incoming or all are outgoing.
 
If all edges incident to a vertex
 all are outgoing (incoming) we say, that the vertex is of
 the first (second) type.

In the middle of any edge of $G^{\prime}$ we put a small cross,
it will be a future crossing of the knot diagram. Now we draw a circle
with the center in each vertex, such that the circles
with centers in the ends of the same edge are tangent
at the small cross. We equip each circle with the orientation
induced by the orientation of the vertex.
These circles will be the Seifert circles for our knot diagram.

Now we form the knot diagram from the Seifert circles by an
algorithm, which is inverse to the Seifert algorithm.
Overcrossings and undercrossings
are defined as follows: if the knot strand goes from a
vertex of the first type to a vertex of the second type, we have an
overcrossing; if the strand goes from a vertex of the second type to
a vertex of the first type, we have an undercrossing.

Note, that even after inserting vertices of valence 2, the graph
has no edge connecting different vertices of valence 2, and thus
the resulting knot diagram has not more than two neighbored crossings
in each neighbored equivalence class.

Planar 3-valent graphs of genus $g$ with Bieulerian path such that
the corresponding diagram has $6g-3$ $\sim$-equivalence classes
 are described in the following theorem.

\begin{theorem}[\cite{SV}]
Let $G$ be a planar 3-valent graph (with Bieulerian path) and $D$ its
knot diagram (as constructed earlier).
Then the following conditions are equivalent:
\begin{enumerate}
\item $G$ is 3-connected (i.e., removing any pair of edges
does not disconnect it),
\item $D$ has $6g-3$ $\sim$-equivalence classes,
\item $D$ admits no (non-trivial) flypes.
\end{enumerate}
\end{theorem}

\begin{rem}\label{remb}
By a theorem of Whitney each 3-valent 3-connected graph
has, if any, a unique planar embedding up to moves in $S^2$
(see \cite{Bar-Natan}). Thus for the cases that are of interest
to us we do not need to care about ambiguities of the planar
embedding, and can consider the graph also abstractly.
\end{rem}

\begin{corr}[\cite{SV}]
There is a bijective correspondence between genus $g$ diagrams
with $6g-3$ $\sim$-equivalence classes and planar 3-connected 3-valent
graphs with Bieulerian paths (considered up to moves in $S^2$ on the
graph and cyclic permutations of the path). 
\end{corr}

{\bf Proof of Theorem 4.11.} 

We put together the previous results.
Clearly, we need to consider only genus $g$ generators $D$ of standard knots, since they have the maximal number of
$\sim$-equivalence classes. By lemma 5.2, this maximal number is $6g - 3$, and generators with that many $\sim$-equivalence classes have graphs with vertices of valence 2 and 3. By theorem 5.3 the diagrams of such graphs are standard, and we know from \cite{SV} that for any crossing number parity, at least one such example exists. Finally, from Part 3 of theorem 5.5 we know that diagrams in the series of D have only symmetries coming from the Bieulerian path, and the order of such a symmetry is at most 6, see \cite{BacherVdovina}.

\section{Connection with Wicks forms}

 An {\it oriented Wicks form\/} is a cyclic word $w= w_1w_2\dots w_{2l}$
 (a cyclic word is the orbit of a linear word under cyclic permutations)
 in some alphabet $a_1^{\pm 1},a_2^{\pm 1},\dots$ of letters
 $a_1,a_2,\dots$ and their inverses $a_1^{-1},a_2^{-1},\dots$, such that
\begin{itemize}
\item[(i)] if $a_i^\epsilon$ appears in $w$ (for $\epsilon\in\{\pm 1\}$)
 then $a_i^{-\epsilon}$ appears exactly once in $w$,
\item[(ii)] the word $w$ contains no cyclic factor (subword of
 cyclically consecutive letters in $w$) of the form $a_i a_i^{-1}$ or
 $a_i^{-1}a_i$ (no cancellation),
\item[(iii)] if $a_i^\epsilon a_j^\delta$ is a cyclic factor of $w$ then
 $a_j^{-\delta}a_i^{-\epsilon}$ is not a cyclic factor of $w$ 
(substitutions of the form $a_i^\epsilon a_j^\delta\longmapsto x,
 \enspace a_j^{-\delta}a_i^{-\epsilon}\longmapsto x^{-1}$ are
 impossible).
\end{itemize}

 An oriented Wicks form $w=w_1w_2\dots$ in the alphabet $A$ is
 {\em isomorphic\/} to $w'=w'_1w'_2\dots$ in an alphabet $A'$ if
 there exists a bijection $\varphi:A\longrightarrow A'$ with
 $\varphi(a^{-1})=\varphi(a)^{-1}$ such that $w'$ and
 $\varphi(w)=\varphi(w_1)\varphi(w_2)\dots$ define the
 same cyclic word.

 The {\em genus\/} $g_t(w)$ of an oriented Wicks
 form $w=w_1\dots w_{2l-1}w_{2l}$ is defined as the topological
 genus of the oriented compact connected surface $S(w)$ obtained
 as described in Section 5.
 Knots diagrams give rise to Wicks forms in the following way.
 
Knots (smooth embeddings of ${\mathbb S}^1$ to ${\mathbb R}^3$) are usually presented
by knot diagrams that are generic immersions of ${\mathbb S^1}$ to 
${\mathbb R}^2$-plane enhanced by information of over-passes and under-passes at the double points. To correspond a Wicks form to a knot diagram $D$ one assigns a letter 
to each double point of the immersion, and the preimages of each double point are denoted by this letter with opposite exponents, 1 and -1.

It was shown in \cite{STV} that for  alternating knots the genus of a knot and the genus of the corresponding Wicks form coniside.

Let $G$ be a cubic (3-valent) connected graph on $4g-2$ vertices and
the word $U$ is one of its Bieulerian paths. We will call them {\em cubic Wicks forms}.
Note that a Bieulerian path can be presented as an oriented Wicks form of genus $g$.

\begin{defi} Wicks forms, which came from Bieulerian paths
of 3-connected planar cubic graphs
on $4g-2$ vertices will be called
{\em planar Wicks forms}. 
\end{defi}

These forms are also maximal in the
sense of \cite{BacherVdovina}.

\begin{defi}
A vertex $V$ (with oriented edges $a,b,c$ pointing toward $V$) in a cubic
Wicks form $w$ is 
{\it positive} if
$$w=ab^{-1}\dots bc^{-1}\dots ca^{-1}\dots \quad {\rm or }\quad 
w=ac^{-1}\dots cb^{-1}\dots ba^{-1}\dots $$
and $V$ is {\it negative} if    
$$w=ab^{-1}\dots ca^{-1}\dots bc^{-1}\dots \quad {\rm or }\quad 
w=ac^{-1}\dots ba^{-1}\dots ab^{-1}\dots \quad 
.$$

\end{defi}

\begin{theorem}(\cite{BacherVdovina})
The number of positive vertices in a genus $g$ cubic Wicks form is $2g-2$, and the number of negative vertices is $2g$.

\end{theorem}

\begin{defi}
Let $a$ be a letter of a Wicks form $W$. If we replace $a$ by a word $a_1...a_k$
(and its inverse by $a_k^{-1}...a_1^{-1}$), we will say that we extended the letter $a$.
\end{defi}

\begin{defi}
A word $V$ is an extended Wicks form, if it is obtained from a Wicks from $W$ by several extensions of letters.
\end{defi}

\section{Isomorphism problem for standard knots}

From now on we consider extended Wicks forms as cyclic words, written on boundaries of discs and 
the graph $\Gamma$ is associated to an extended Wicks form $W$, the word $W$ is written with letters of an alphabet $\alpha$.

\begin{theorem}\label{th:7.2}
The isomorphism problem of standard genus $g$ knots with $n$ crossings given by their alternating diagrams
is equivalent to the isomorphism problem of extended Wicks forms of genus $g$ and length $2n$.
\end{theorem}
\begin{proof}
Let $K_1$ and $K_2$ be two standard knots given by their alternating diagrams.
First consider the case when both $K_1$ and $K_2$ are generating diagrams.
By \cite{MenThis} any two diagrams of alternating knots can be obtained from one another
by a sequence of flypes. By the Theorem 5.5, \cite{SV}, alternating diagrams of generating diagrams do not admit flypes,
so the isomorphism class of a standard generating knot is uniquely determined by its alternating diagram.
Alternating diagrams of $K_1$ and $K_2$ uniquely determine Gau\ss{} diagrams $D_1$
and $D_2$, and Gau\ss{} diagrams $D_1$ and $D_2$ uniquely determine two quadratic words
$W_1$ and $W_2$ (see the beginning of Chapter 4). By the Theorem 5.7 and explicit description
of standard alternating knots, all the vertices of the graphs $\Gamma_1$ and $\Gamma_2$
corresponding to $W_1$ and $W_2$ have valencies two or three, so $W_1$ and $W_2$
are extended Wicks forms.

Two diagrams of a(n alternating) knot in the same generating series are transformable into each other by a flype the generating diagram, see \cite{SV}, p.10-11. Since the generating diagrams we consider
do not admit flypes, the isomorphism type is uniquely defined by an extended Wicks form, as before.

\end{proof}

\section{Computational complexity}
In this section we will construct low complexity algorithms to solve some problems about strictly quadratic words in a free group and then use these algorithms to prove Theorems \ref{th:genus} and \ref{th:iso}.
\begin{prop} \label{prop:iso} There exists an  algorithm with time complexity $n\log n$  that given two strictly quadratic cyclically reduced words $w$, $w_1$ of length $2n$ in the free group $F(X)$ determines if there is a permutation $\sigma$ of the letters in $X$ such that $w_1(\sigma X)$  is a cyclic permutation of $w(X)$.
\end{prop}
\begin{proof}  The word $w$ of length $2n$ will be represented as an array (a string of pairs) $\mathcal W$ such that  $W_m$ is a pair consisting of a letter in position $m$ of $w$ (say, $a_m$ or $a_m^{-1}=A_m$) and number $m$.
The second array $\mathcal V$ consists of $n$ triples indexed by $n$ letters $a_i$. A triple $V_{a_i}$ consists of $a_i,$ position $m$ (in binary) of $a_i$ and position $k$ (in binary) of $A_i$.  

The algorithm scans the array $\mathcal W$ and creates the array $\mathcal V$. Scanning the pair $(a_i,m)$ or $(A_i, m)$ it puts $m$ in the second position  of $V_{a_i}$ if this position has not been filled in yet, or in the third position if the second position has  already been filled.  For each triple $(a_i, m,k)$ we compute $d_m=k-m$ and $d_k=k+2n-m$.  We construct a sequence $d_1(w),\ldots ,d_{2n}(w)$. Since we need $\mathcal O(\log n) $ time to subtract two numbers that are less or equal to $2n$, this takes  time $\mathcal O (n\log n).$

Now $w_1(\sigma X)$  is a cyclic permutation of the word $w(X)$ if and only if the sequence 
$d_1(w_1),\ldots ,d_{2n}(w_1)$ is a cyclic permutation of $d_1(w),\ldots ,d_{2n}(w)$. This can be decided in linear time in $n$ using  Knuth-Morris-Pratt  substring searching algorithm \cite{KMP}, \cite{Mat}.  This algorithm searches for occurrences of a word  within a  text string  by employing the observation that when a mismatch occurs, the word itself contains sufficient information to determine where the next match could begin, thus bypassing re-examination of previously matched characters.

\end{proof}

\begin{prop} \label{prop:genus} There exists an  algorithm with linear time complexity  that computes the genus of a  strictly quadratic cyclically reduced word  of length $2n$ in the free group $F(X)$.  
\end{prop}
\begin{proof} Let $\mathcal W$ and $\mathcal V$ be the arrays defined in the proof of Proposition \ref{prop:iso}. 
We now define the graph $\Delta$  with vertices numbered from $0$ to $2n-1$ and edges obtained from the array $\mathcal V$. For each triple $V_{a_i}=(a_i, m, k)$ there will be two edges $(m,k+1)$ and $(k, m+1).$ Each connected component of $\Delta$ represents one vertex of the graph $\Gamma$. Connected components can be of the form 
$\{i,j\}$ or $\{i,j,t\}$ or $\{i_1\ldots i_k\}, k\leq 2n.$  BFS or DFS algorithms find connected components from the list of edges in
time $\mathcal O (|E|)$, therefore in time $\mathcal O(n)$.  Now we know the number of vertices $|V|$ of the graph $\Gamma$ and we know that it has $n$ edges. Therefore the Euler characteristic is $\kappa =|V|-n+1$ and the genus $g= \frac{1}{2} (2-\kappa ).$



\end{proof}

We will recall definitions of some other complexity classes that we will consider.

{\em Logspace} (denoted L) is the class of functions computable by a deterministic Turing
machine with working tape bounded logarithmically in the length of the input. There are two more tapes, the input tape, where we can only read but cannot write, and the output tape where we can write but cannot read while working.

For every $n, m \in N$ a Boolean circuit $C$ with $n$ inputs and $m$ outputs 
is a directed
acyclic graph. It contains $n$ nodes with no incoming edges; called the input nodes
and $m$ nodes with no outgoing edges, called the output nodes. All other nodes
are called gates and are labeled with one of $\vee$, $\wedge$ or $\neg$ (in other words, the logical
operations OR, AND, and NOT). The  $\vee$, $\wedge$ nodes have fanin (i.e., number of
incoming edges) of 2 and the $\neg$ nodes have fanin 1. The size of $C$, denoted by $|C|$,
is the number of nodes in it.
The circuit is called a Boolean formula if each node has at most one outgoing edge.

 A $TC^0$
circuit with $n$ inputs is a boolean circuit of constant
depth using NOT gates and unbounded fan-in AND, OR, and MAJORITY
gates, such that the total number of gates is bounded by a polynomial function
of $n$. A MAJORITY gate outputs 1 when more than half of its inputs are 1. A
function $f(x)$ is $TC^0$
-computable (more casually, an algorithm is in $TC^0$) if for
each $n$ there is a $TC^0$
circuit $F_n$ with n inputs which produces $f(x)$ on every
input $x$ of length $n$. The composition
of two $TC^0$
-computable functions is again $TC^0$
-computable.
Since this definition of being computable only asserts that such a family
${F_n}_{n=1}^{\infty}$
of circuits exists, one normally imposes in addition a uniformity condition
stating that each $F_n$ is constructible in some sense. We will only be
concerned here with standard notion of DLOGTIME uniformity, which asserts
that there is a random-access Turing machine which decides in logarithmic time
whether in circuit $F_n$ the output of gate number $i$ is connected to the input of
gate $j$, and determines the types of gates $i$ and $j$. We refer the reader to \cite{Vol99}
for further details on $TC^0$. The problems Iterated Addition, Iterated Multiplication,
Integer Division are all in $TC^0$
no matter whether inputs are given in unary or binary \cite{Vol99}. 

The relation between the classes is as follows:

$$TC^0\subseteq L\subseteq P, $$
where $P$ denotes the class of problems solvable in polynomial time.
\begin{prop} \label{prop:genus1} \label{prop:genus2} There exists a Logspace algorithm that computes the genus of a  strictly quadratic cyclically reduced word  in the free group $F(X)$.  \end{prop}
\begin{proof} The genus of the word $w(a_1,\ldots a_n)$ is the genus of the surface $S$ that one obtains when glues together edges with the same label of the polygon with boundary label $w$.  The word $w$ becomes the label of the graph $\Gamma$ on the surface. We will write the word $w$ as a cyclic permutation.
After multiplying this permutation by $\Pi _{i=1}^n(a_i,a_i^{-1})$ we obtain a permutation $\pi$ such that the cycles 
of $\pi$ correspond to the vertices of $\Gamma$, see \cite{VL}, Section 2, and, therefore, can compute the Euler characteristic and the genus of the surface $S$. The algorithm to represent the product of two permutations given by their cyclic representation also as a cyclic representation belongs to Logspace by \cite{CM}, Theorem 3 (this is the problem PP2).
\end{proof}

\begin{prop} \label{prop:genus1} There exists a $TC^0$ algorithm that computes the genus of a  strictly quadratic cyclically reduced word $w$ in the free group $F(X)$ corresponding to the standard  alternating knot. \end{prop}
\begin{proof} In this case the product of the involution $\Pi$ and the cycle $(w)$ corresponding to $w$   does not have cycles longer than $3$.  
 We encode a  word $w$ of length $2n$ 
as the array $\mathcal W$ from the proof of Proposition \ref{prop:iso}. The edges of the graph $\Delta$ represent the permutation $\sigma$ on $n$ elements presented pointwise, as the set of pairs $(k,m)$, such that $\sigma (k)=m$.
We can sort (sorting in in $TC^0$) these pairs according to the order of the first component. To find the Euler characteristic we have to determine the number of cycles in the permutation $\sigma$. Since the  knot is standard we know that the maximal length of a cycle is three.

Let the second level array  have cells $\delta _{m,k}$, where $\delta _{m,k}$ contains the  pair $(m,k)$ if it is an edge of $\Delta$  and contains zero otherwise.  

In the next level array we will have a  triple  $(i,j,k)$ for each pair of edges  $(i,j)$ and $(j,k)$ in $\Delta$, a pair $(i,j)$  
for each pair $(i,j), (j,i)$ and zero for pairs $(i,j),(k,m)$ where $i,j,k,m$ are distinct.
Then we divide the number of different triples by three and add the number of pairs. This is the number of connected components in $\Delta$ and number of vertices in $\Gamma$.

\end{proof}

\begin{prop} \label{prop:iso1} There exists a $TC^0$ algorithm that given two strictly quadratic cyclically reduced words $w_1$, $w_2$ in the free group $F(X)$ determines if there is a permutation $\sigma$ of the letters in $X$ such that $w_2(\sigma X)$  is a cyclic permutation of $w(X)$. Therefore this problem is also in $L$.

\end{prop}\begin{proof}
To decide if two words of length $n$ differ by a permutation of letters and a cyclic permutation we organize the binary circuit as follows. We encode both cyclic word $w_1, w_2$ of length $n$ represented as a labelled cycle graph, as a set of triples of natural numbers (encoded as binaries) as above. We also take all cyclic permutations $w_{2k}$ of $w_2$ by adding $1$ to the first two entries of corresponding triples. For $w_1, w_{2k}$ we define $\delta _{ij}(w_1)$,      $\delta _{ij}(w_{2k})$ as above. Then we compare  $\delta _{ij}(w_1)$ with each      $\delta _{ij}(w_{2k})$ in parallel.
If for some $k$ for all $i,j$ $\delta _{ij}(w_1)=\delta _{ij}(w_{2k})$, then $w_1$ and $w_2$ differ by a permutation of letters and a cyclic permutation, otherwise they do not.

 \end{proof}

{\bf Proof of Theorem \ref{th:genus}} 

It was shown in \cite{STV} that for the alternating knots the genus of the knot and the genus of corresponding Wicks form coniside. The genus of a Wicks form is the same as the genus of the extended Wicks form. The statement of the theorem now follows from Propositions \ref{prop:genus}, \ref{prop:genus2} and \ref{prop:genus1}.  

{\bf Proof of Theorem \ref{th:iso}}

By Theorem \ref{th:7.2}, the isomorphism problem of standard genus $g$ knots with $n$ crossings given by their alternating diagrams
is equivalent to the isomorphism problem of extended Wicks forms of genus $g$ and length $2n$. Therefore the statement of the theorem follows from Propositions \ref{prop:iso} and \ref{prop:iso1}.

{\bf Acknowledgment.}  
The first author was supported by the PSC-CUNY  award, jointly funded by The Professional Staff Congress and The City University of New York and by a grant 461171 from the Simons Foundation. 
The second author thanks Hunter College CUNY for the Peluso Professorship award for 2017-2018, when most of the work on this paper was done. 
We thank I. Agol,  M. Lackenby, S. Vassileva and J. Macdonald for useful discussions and UC Berkeley and INI Cambridge for providing support.

\end{document}